\documentclass[12pt]{amsart}
\usepackage{graphicx} 
\usepackage[a4paper, left=3.5cm, right=3.5cm, top=3cm]{geometry}
\usepackage[ english]{babel}
\usepackage{comment}
\usepackage[T1]{fontenc}
\usepackage{tabularx}
\usepackage{amsfonts}
\usepackage{amsmath, amssymb,amsthm}
\usepackage{mathtools}  
\usepackage{relsize}
\usepackage{hyperref}
\usepackage[utf8]{inputenc}
\usepackage{enumerate, enumitem}

\newcommand{\R}{\mathbb{R}} 
\newcommand{\N}{\mathbb{N}} 
\newcommand{\var}{\varepsilon} 
\DeclareMathOperator{\Div}{div}

\newcommand{\lorentz}{L^{2,1}}

\newcommand{\loc}{\text{loc}}

\newtheorem{theorem}{Theorem}[section]

\newtheorem{lemma}[theorem]{Lemma}
\newtheorem{remark}[theorem]{Remark}

\title[Energy identity/ no-neck in homogeneous spaces]{Energy identity and no neck property for $\var$-harmonic and $\alpha$-harmonic maps into homogeneous target manifolds}
\author{Carolin Bayer}
\address[C.~Bayer]{This work was carried out at Karlsruhe Institute of Technology. The author is now with Department of Mathematics \\
ETH Zentrum \\
CH-8093 Zurich \\ Switzerland}
\email{carolin.bayer@math.ethz.ch}
\author{Andrew M. Roberts}
\address[A.~M.~Roberts]{School of Mathematics\\ University of Leeds\\ Leeds, LS2 9JT\\ United Kingdom}
\email{whgt4010@leeds.ac.uk}
\date{\today}
\thanks{We would like to thank our supervisors Tobias Lamm and Ben Sharp for their enriching discussions and continuous guidance throughout this collaboration. \\
The first author is funded by Deutsche Forschungsgemeinschaft (DFG) - RTG 2229, Project number 281869850 and by the Swiss National Fund, Project SNF $200020\_219429$.\\
The second author would like to thank Karlsruhe Institute of Technology for their kind hospitality during the project.}
\begin{document}

\maketitle

 \begin{quote}
  \footnotesize
  \textsc{Abstract.}  
 In this paper we show the energy identity and the no-neck property for $\varepsilon$- and $\alpha$-harmonic maps with homogeneous target manifolds. To prove this in the $\varepsilon$-harmonic case we introduce the idea of using an equivariant embedding of the homogeneous target manifold. 
\end{quote}

\vspace{+0.8cm}


\normalsize

\section{Introduction}

Let $(M^2,g)$ be a smooth, compact Riemannian surface without boundary and $(N^n,h)$ be a compact Riemannian manifold without boundary which we assume to be equipped with an isometric embedding into Euclidean space $N\hookrightarrow\R^l$.
Define for $u\in W^{1,2}(M,N)$ the Dirichlet energy
\begin{equation*}
    E_0[u]=\frac{1}{2}\int_{M}\vert \nabla u\vert ^2.
\end{equation*}
We call $u$ harmonic if it is a critical point of the energy functional and these maps satisfy
\begin{equation*}
    \Delta u = A(u)(\nabla u,\nabla u),
\end{equation*}
where $A(u)(X,Y)=(\nabla_XY)^\perp\in (T_uN)^\perp$ is the second fundamental form of the embedding $N \hookrightarrow \R^l$.
The Dirichlet energy does not obey the Palais-Smale condition, so various approximations to the energy functional have been introduced which allow us to use methods from the calculus of variations.
In \cite{Sacks-Uhlenbeck} Sacks and Uhlenbeck introduced the notion of $\alpha$-energy defined for $\alpha>1$ by
\begin{equation}
    \bar{E}_\alpha[u]=\frac{1}{2}\int_{M} \big((1+|\nabla u|^2)^\alpha-1\big) \label{eq: alps_harmonic_energy}
\end{equation}
and in \cite{Lamm} Lamm introduced the notion of $\var$-energy defined for any $\varepsilon>0$ by
\begin{equation}
    \tilde{E}_\varepsilon[u]= \frac{1}{2}
    \int_M \big( |\nabla u|^2+\varepsilon|\Delta u|^2 \big). \label{eq: eps_harmonic_solution}
\end{equation}
Both these functionals obey the Palais-Smale condition so critical points exist, and can be shown to be smooth. We call these critical points $\alpha$-harmonic (resp. $\var$-harmonic).
Note that the Euler-Lagrange equation of \eqref{eq: alps_harmonic_energy} can be written as
\begin{align}
    \Div(F_\alpha  \nabla u_\alpha ) &=  F_\alpha A(u_\alpha)(\nabla u_\alpha , \nabla u_\alpha)  ,\label{eq: E-L2}
\end{align} where $F_\alpha := (1+\vert \nabla u_\alpha \vert^2)^{\alpha-1}$.
The Euler-Lagrange equation of \eqref{eq: eps_harmonic_solution} is given by
\begin{align*}
    \Delta u - \var \Delta^2 u 
    &= \sum\limits_{i=n+1}^l \var \big( \Delta (\langle \nabla u, \nabla\nu_i|_u(\nabla u) \rangle) + \Div(\langle \Delta u, \nabla\nu_i|_u(\nabla u) \rangle) \nonumber \\
    & \quad + \langle \nabla \Delta u , \nabla\nu_i|_u(\nabla u) \nabla u \rangle \big) \nu_i (u) - A(u)(\nabla u, \nabla u),
     \label{eq: pde_epsilon_harmonic}
\end{align*}
where $\{ \nu_i \}_{i=n+1}^l$ is a smooth local orthonormal frame for the normal space of $N \subset \R^l$ near $u(x)$.
Note that this is equivalent to
\begin{equation}\label{eq: E-L for epsilon}
    (\Delta u - \var \Delta^2 u)^\top=0 .
\end{equation}

If we now take a sequence $u_k$ of $\alpha_k$-harmonic (resp. $\var_k$-harmonic) maps with $\alpha_k\rightarrow1$ (resp. $\var_k\rightarrow0$) with uniformly bounded $\alpha_k$-energy (resp. $\varepsilon_k$-energy) then in both cases it is clear that there exists some weakly harmonic map $u_0$ such that $u_k\rightharpoonup u_0$ weakly in $W^{1,2}(M,N)$. Further, $u_0$ will be smooth with $C^m_{\loc}$ convergence for any $m$ away from a finite set $\Sigma$ and at the points in $\Sigma$, in both cases, the maps will undergo the standard bubbling procedure.\\

The two main questions that arise from the bubbling procedure are those of potential energy loss and neck formation. We would like to show that there is no energy loss or oscillation in the neck region. There are many previous results related to this type of energy approximation, some of which we will now discuss. In \cite{Parker_bubble} Parker showed that sequences of harmonic maps into compact Riemannian manifolds also obey the same bubbling procedure and have both the no energy loss and no neck properties, in \cite{Lamm} Lamm showed that sequences of $\var$-harmonic maps into round spheres have the no energy loss property and in \cite{noneck_li_zhu} Li and Zhu showed that sequences of $\alpha$-harmonic maps into round spheres have the no energy loss and the no neck properties. Their result is fundamentally based on the duality of $(L^{2,1})^* = L^{2,\infty}$ 
\begin{align*}
  \Bigg\vert  \int\limits_U f \cdot g \; \Bigg\vert \leq \Vert f\Vert_{L^{2,1}(U)} \cdot \Vert g \Vert_{L^{2,\infty}(U)}.
\end{align*}
The importance of this duality and its applications to so-called neck analysis was first stressed and exploited by Lin and Rivière in \cite{lin_riviere_l21_idea} where they derive an energy-defect identity for sequences of stationary harmonic maps from higher dimensional domains into spheres. This result was systematised by Rivière and Laurain in \cite{laurain_riviere_2011} for elliptic linear systems with antisymmetric potentials in two dimensions.
\vspace{0.2cm}

Regarding the regularised energy $\bar{E}_\alpha$, Chen and Tian gave in \cite{chen} the energy identity for minimising sequences which map into general target manifolds for a given homotopy class. Furthermore, Li and Wang proved in \cite{LiWang_weak_energy} a generalised energy identity for $\alpha$-harmonic maps again into general target manifolds and studied the length of necks. In particular, in the case where there is only one bubble, they gave an explicit length formula for the neck. In addition, Lamm \cite{lamm_entropy_condition} proved under certain entropy-type condition that for $\alpha$- and $\var$- harmonic sequences no energy loss occurs. He studied Hopf-differential-type structures in the two approximations. Da Lio and Rivi\`ere reformulated in \cite{dalio_riviere} for $p$-harmonic systems with anti-symmetric potentials their divergence form as a conservation law. They provided a sufficient condition for controlling the length of the necks at the limit to a much
wider case of PDE’s including p-harmonic relaxation of arbitrary conformally invariant Lagrangians.\\

Jost, Liu and Zhu \cite{jost_liu_zhu_general} consider a sequence of maps from a compact Riemann surface with smooth boundary which map into a general compact Riemannian manifold with free boundary on a smooth submanifold. In addition, this sequence has a uniformly bounded energy and tension field. They then showed that the energy identity and the no-neck property hold during a blow-up process. Under these assumptions, they considered the Sacks-Uhlenbeck $\alpha$-harmonic case in \cite{jost_liu_zhu_alpha} and proved a generalised energy identity. Further, for $\alpha$-harmonic maps in the spherical target manifold case, Li, Zhu and Zhu \cite{li_zhu_zhu_sphere} showed that the no-neck property additionally holds. \\

We will now show both the no energy loss and no neck property for sequences of $\alpha$-harmonic maps and of $\var$-harmonic maps in the case where the target manifold is a compact homogeneous manifold. A Riemannian manifold $(N,h)$ is defined to be homogeneous if the action of its Lie group of isometries $G$ is transitive. In \cite{Helein_paper} H\'elein introduced the idea of using Killing vector fields induced by the Lie algebra of $G$ to work with homogeneous manifolds. For any $a\in\mathfrak{g}$ we have a Killing vector field $\rho(a)\in\Gamma(TM)$ induced by $a$.\\

In the case of $\alpha$-harmonic maps, for any $\psi\in G$ we can show that \linebreak $\bar{E}_\alpha[\psi\circ u]=\bar{E}_\alpha[u]$, due to the fact that $\bar{E}_\alpha$ is an intrinsic property. Then by Noether's theorem we obtain a conservation law similar to the one derived in \cite{noneck_li_zhu} for the sphere, following H\'elein's approach for harmonic maps. Then by modifying their arguments one can show the no energy loss and no neck property.\\

However, this does not work in the case of $\var$-harmonic maps, since $\tilde{E}_\var[\psi\circ u]$ generally does not equal $\tilde{E}_\var[u]$. The Laplacian depends on the embedding of $N$, so $|\Delta u|$ is generally not equal to $|\Delta( \psi\circ u)|$. To account for this, we will use a specific type of embedding, the existence of which was shown in \cite{Moore_equivariant} by Moore.
\begin{theorem}\label{theorem: equivariant embedding}
    Any homogeneous Riemannian manifold $(N,h)$ admits an isometric and equivariant embedding into some Euclidean space.
\end{theorem}
In particular we have that if $(N,h)$ is a homogeneous Riemannian manifold with $G$ its Lie group of isometries, there exists $\Phi:N\rightarrow\R^l$ an isometric embedding and $\Pi:\text{Isom}(N)=G\rightarrow O(l)$ an embedding such that
\begin{equation*}
(\Pi(\psi))(\Phi(q))=\Phi(\psi(q))
\end{equation*}
for any $\psi\in G,q\in N$. In other words, any intrinsic isometry of $N$ is induced by some extrinsic isometry of the entire Euclidean space.
In this case, we easily get $\tilde{E}_\var[\Pi(\psi)\circ u]=\tilde{E}_\var[u]$. Then Noether's theorem gives us a conservation law which we will then use to obtain the two desired properties. In this case we further find that the $\rho(a)$ used by H\'elein will be explicitly of the form $\rho(a)(q)=\eta_a(q)$ for $\eta_a$ some fixed anti-symmetric matrix independent of $q$, which will considerably ease computations.\\

In particular, we obtain the following.
\begin{theorem}[Conservation law for $\varepsilon$-harmonic maps]\label{theorem: conservation law}
    Let $(M,g)$ be a Riemannian surface, $(N,h)\hookrightarrow\R^l$ be a homogeneous space isometrically and equivariantly embedded into Euclidean space with $\Pi:G\rightarrow O(l)$ the associated Lie group embedding.
     $u \in C^\infty(M,N)$ is a solution of \eqref{eq: E-L for epsilon} if and only if, for any $\eta$ of the form $\eta = \frac{\partial}{\partial t}\Pi(\gamma(t)) \big|_{t=0}$ for some $\gamma(t)$ a smooth path in $G$ with $\gamma(0)$ the identity, we have
    \begin{align*}
         \Div_g \big( \langle du, \eta u \rangle -  \varepsilon  d (\langle \Delta_g u ,  \eta u \rangle)+ 2\var \langle \Delta_g u, \eta du \rangle \big) =0.   \label{eq: conservation law intrinsic}
     \end{align*}
\end{theorem}

\begin{remark}
    In the case of a harmonic map, we have $\varepsilon=0$, this then allows us to rewrite H\'elein's conservation law from \cite{Helein_paper} as
    \begin{equation*}
        \Div_g\langle du,\eta u\rangle=
    \langle du,\eta du\rangle + \langle\Delta_g u,\eta u\rangle=0
    \end{equation*}
    in the case the target is equivariantly embedded into Euclidean space. 
\end{remark}

The following theorem was proved by Lamm in \cite{Lamm}.
\begin{theorem}[Bubbling for $\var$-harmonic maps] \label{theorem: BubblingEps}
    Let $(M^2,g)$ be a smooth, compact Riemannian surface without boundary and let $(N^n,h)\hookrightarrow\R^l$ be a compact Riemannian manifold isometrically embedded into Euclidean space.  Further, $u_\var \in C^\infty(M,N) \; (\var\to 0)$ is a collection of critical points of $\tilde{E}_\var$ with uniformly bounded $\var$-energy.\\
    Then there exists a sequence with $\var_k \to 0$ and at most finitely many points $x_1,...,x_p \in M$ such that 
    \begin{align*}
        u_{\var_k} &\rightharpoonup u_0 \text{ in } W^{1,2}(M,N), \\
        u_{\var_k} &\to u_0 \text{ in } C_{loc}^m(M \setminus \{x_1,...,x_p \}) \; \forall m \in \N,
    \end{align*}
    where $u_0 \in C^\infty(M,N)$ is a smooth harmonic map.\\
    Further performing a blow-up at each $x_i, 1\leq i \leq p$, there exist at most finitely many non-trivial smooth harmonic maps $\omega^{i,j}: S^2 \to N, 1\leq j \leq j_i$, sequences of points $x_k^{i,j}\in M, x_k^{i,j} \to x_i$, and sequences of radii $t_k^{i,j} \in \R_+$, $t_k^{i,j}\to 0$, such that 
\begin{align*}
    \max\Bigg\{ \frac{t_k^{i,j}}{t_k^{i,j'}} ,   \frac{t_k^{i,j'}}{t_k^{i,j}}, \frac{\text{dist}(x_k^{i,j}, x_k^{i,j'})}{t_k^{i,j}+ t_k^{i,j'}} \Bigg\} &\to \infty \quad
    \forall  1 \leq i \leq p, \; 1 \leq j,j' \leq j_i, \; j \neq j', \nonumber \\
    \frac{\var_k}{(t_{k}^{i,j})^2} &\to 0 \quad \;\; \forall 1 \leq i \leq p, \; 1 \leq j \leq j_i.
\end{align*}
\end{theorem}
We will prove the following.
\begin{theorem}[Energy identity and no neck property for $\var$-harmonic maps]\label{theorem: energy+neck eps}
Under the conditions of Theorem \ref{theorem: BubblingEps} with the additional assumption that \linebreak $(N^n,h) \hookrightarrow \R^l$ is a homogeneous space isometrically and equivariantly embedded into Euclidean space, we have the following.
\begin{itemize}[leftmargin=*]
    \item \textbf{The energy identity}
    \begin{align}
        \lim_{k\to \infty} \tilde{E}_{\var_k}[u_{\var_k}] = E_0[u_0] + \sum\limits_{i=1}^p \sum\limits_{j=1}^{j_i} E_0[\omega^{i,j}] \nonumber
    \end{align}
    and 
    \begin{align*}
        \int\limits_{B_{R_0}(x_i)} \big( \vert \nabla\omega_k^{j_i} \vert^2 + \var_k \vert \Delta \omega_k^{j_i} \vert^2 \big) \to \int\limits_{B_{R_0}(x_i)} \vert \nabla u_0 \vert^2, \quad \forall 1 \leq i \leq p,
    \end{align*}
where $ \omega_k^{j_i}= u_{\var_k}- \sum\limits_{j=1}^{j_i} (\omega^{i,j}(\frac{\cdot - x_k^{i,j}}{t_k^{i,j}})- \omega^{i,j}(\infty))$ and $0<R_0 < \frac{1}{2} \min\big\{ \text{inj}(M),$ \\
$ \min \{ \text{dist}(x_i,x_j) \;|\; 1\leq i \neq j \leq p\} \big\}$ is some number and $\infty$ is identified with the north pole of $S^2$ by stereographic projection. 
    \item  \textbf{The no-neck property}
    \begin{align*}
        \lim_{k \to 0} \Big\Vert u_{\var_k}(\cdot) - u_0(\cdot) - \sum\limits_{i=1}^p \sum\limits_{j=1}^{j_i} \big[\omega^{i,j} \big(\frac{\cdot - x_j^{i,j}}{t_k^{i,j}} \big) - \omega^{i,j}(\infty)\big] \Big\Vert_{L^\infty} = 0.
    \end{align*}
\end{itemize}
\end{theorem}

The following theorem was proved by Sacks and Uhlenbeck in \cite{Sacks-Uhlenbeck}.
\begin{theorem}[Bubbling for $\alpha$-harmonic maps] \label{theorem: BubblingAlps}
    Let $(M^2,g)$ be a smooth, compact Riemannian surface without boundary and let $(N^n,h)\hookrightarrow\R^l$ be a compact Riemannian manifold isometrically embedded into Euclidean space. Further, $u_\alpha \in C^\infty(M,N)$  $ (\alpha\to 1)$ is a collection of critical points of $\bar{E}_\alpha$ with uniformly bounded $\alpha$-energy.\\
    Then there exists a sequence for $\alpha_k \to 1$ and at most finitely many points $x_1,...,x_p \in M$ such that 
    \begin{align*}
        u_{\alpha_k} &\rightharpoonup u_0 \text{ in } W^{1,2}(M,N) \\
        u_{\alpha_k} &\to u_0 \text{ in } C_{loc}^m(M \setminus \{x_1,...,x_p \}) \; \forall m \in \N,
    \end{align*}
    where $u_0 \in C^\infty(M,N)$ is a smooth harmonic map.\\
    Further performing a blow-up at each $x_i, 1\leq i \leq p$, there exist at most finitely many non-trivial smooth harmonic maps $\omega^{i,j}: S^2 \to N, 1\leq j \leq j_i$, sequences of points $x_k^{i,j}\in M, x_k^{i,j} \to x_i$, and sequences of radii $t_k^{i,j} \in \R_+$, $t_k^{i,j}\to 0$, such that 
\begin{align*}
    \max\Big\{ \frac{t_k^{i,j}}{t_k^{i,j'}} ,   \frac{t_k^{i,j'}}{t_k^{i,j}}, \frac{\text{dist}(x_k^{i,j}, x_k^{i,j'})}{t_k^{i,j}+ t_k^{i,j'}} \Big\} &\to \infty, \quad
    \forall  1 \leq i \leq p, \; 1 \leq j,j' \leq j_i, \; j \neq j'. \nonumber \\
\end{align*}
\end{theorem}
We will prove the following.
\begin{theorem}[Energy identity and no neck property for $\alpha$-harmonic maps]\label{theorem: energy+neck Alps} 
Under the conditions of Theorem \ref{theorem: BubblingAlps} and the further assumption that $N$ is a homogeneous manifold, we have the following:
\begin{itemize}[leftmargin=*]
    \item \textbf{The energy identity}
    \begin{align}
        \lim_{k\to \infty} \bar{E}_{\alpha_k}[u_{\alpha_k}] = E_0[u_0] + \sum\limits_{i=1}^p \sum\limits_{j=1}^{j_i} E_0[\omega^{i,j}] \nonumber
    \end{align}
    and
    \begin{align*}
        \int\limits_{B_{R_0}(x_i)}  \big( (1+ \vert \nabla\omega_k^{j_i} \vert^2)^{\alpha_k}-1 \big) \to \int\limits_{B_{R_0}(x_i)} \vert \nabla u_0 \vert^2, \quad \forall 1 \leq i \leq p,
    \end{align*}
where $ \omega_k^{j_i}= u_{\alpha_k}- \sum\limits_{j=1}^{j_i} \big(\omega^{i,j}(\frac{\cdot - x_k^{i,j}}{t_k^{i,j}})- \omega^{i,j}(\infty)\big)$ and $0<R_0 < \frac{1}{2} \min\{ \text{inj}(M),$ \\
$ \min \{ \text{dist}(x_i,x_j) \; |\; 1\leq i \neq j \leq p\} \}$ is some number and $\infty$ is identified with the north pole of $S^2$ by stereographic projection. 
    \item  \textbf{The no-neck property}
    \begin{align*}
        \lim_{k \to \infty} \Big\Vert u_{\alpha_k}(\cdot) - u_0(\cdot) - \sum\limits_{i=1}^p \sum\limits_{j=1}^{j_i} \big[\omega^{i,j} \big(\frac{\cdot - x_j^{i,j}}{t_k^{i,j}} \big) - \omega^{i,j}(\infty) \big] \Big\Vert_{L^\infty} = 0.
    \end{align*}
\end{itemize}
\end{theorem}

Finally, Li and Wang constructed in \cite{counterexample} a general target manifold such that the energy identity does not hold in the case of $\alpha$-harmonic maps. By following this construction, we show that their manifold also gives a counterexample for the energy identity in the case of $\var$-harmonic maps. Though we note that this construction relies on varying homotopy classes and so does not say anything of the general case for a fixed homotopy class.
\\

\section{Energy bound for $\varepsilon$-harmonic maps}
In this section, we prove an $L^2$-energy bound for $\var$-harmonic maps in the neck region. The key fact lies in the construction of the conservation law, Theorem \ref{theorem: conservation law}. We will then use this along with the estimates obtained by Lamm in \cite{Lamm} to obtain the energy bound. We also note that the conservation law is of a form similar to the one obtained by Lamm for the spherical case, and one can use his approach to obtain the energy bound in the homogeneous case. We, however, present a different proof, as from our proof the no neck condition will easily follow.
We also note again that this proof will only work in the case that our target is equivariantly embedded.

To prove Theorem \ref{theorem: conservation law} we first need the following lemma.
\begin{lemma}\label{lemma: eta spanning}
 Let $(N,h)\hookrightarrow\R^l$ be a homogeneous space isometrically and equivariantly embedded into Euclidean space with $\Pi:G\rightarrow O(l)$ the associated Lie group embedding. Then there exists a finite collection $(\gamma_i)_{i=1}^I$ of smooth paths in G with $\gamma_i(0)=\text{id}$ such that, setting $\eta_i=\frac{\partial}{\partial t}(\Pi(\gamma_i(t)))\big|_0$, the collection $(\eta_iq)_{i=1}^I$ spans $T_qN$ for any $q\in N$.
\end{lemma}
As our $\eta$ are the same objects as H\'elein's $\rho$, the $\eta$ span for the same reason as the $\rho$ do.
\begin{proof}
    Take some basis $a_1,...,a_I$ of $\mathfrak{g}$, the Lie algebra of $G$, and paths $\gamma_i(t)$ in $G$ with $\gamma_i(0)=$id and $d\gamma_i|_0=a_i$ for $i=1,...,I$.
    Now set $\eta_i=\frac{\partial}{\partial t}(\Pi(\gamma_i))|_0$.
    The $\eta_i q$ will now span the tangent space for each $q$ due to the transitivity of the action of $G$.
    Explicitly given $q\in N$ and $\omega\in T_qN$ we can find some path $\zeta$ in $N$ with $\zeta(0)=q$ and $\zeta'(0)=\omega$. This can then be lifted to some path $\gamma$ in $G$ with $\gamma(0)=\text{id}$ and $\Pi(\gamma(t))q=\zeta(t)$. Now $\omega=\frac{\partial}{\partial t}\Pi(\gamma(t))|_0=d\Pi|_{\text{id}}(d\gamma|_0)$. We know that $d\gamma|_0$ is in the span of the $a_i$ so $\omega$ must be in the span of the $\eta_iq$.
\end{proof}

\begin{proof}[Proof of Theorem \ref{theorem: conservation law}]
    First, to motivate our proof, one can observe that $\tilde{E}_\var$ is invariant under $\Pi(G)\subset O(l)$
    \begin{align*}
       \tilde{E}_\var[\Pi(\gamma(t))\circ u] = \tilde{E}_\var[u] 
    \end{align*}
    for all $t$, so we should have some conservation law by Noether's theorem.
    Now explicitly for any $q\in N$ we have that $\Pi(\gamma(t))q$ is some path in $N$ so $\eta q\in T_qN$. But also $\Pi(\gamma(t))$ is a path in $O(l)$ so $\eta$ is in $\mathfrak{o}(l)$, in particular it can be viewed as an anti-symmetric matrix.
    Further, $u$ being a solution of \eqref{eq: E-L for epsilon} implies that $(\Delta_g u - \var \Delta^2_gu) \in (T_uN)^\perp$, where $\Delta_g$ is the Laplacian with respect to the metric on $M$. We can then observe, noting that $\eta$ is constant,
    \begin{equation}\label{eq: first_step_solution}
    \begin{split}
        0& = \langle \Delta_g u - \var \Delta_g^2 u , \eta u \rangle   \\
        &= \Div_g \big( \langle du, \eta u \rangle -  \varepsilon  d \langle \Delta_g u ,  \eta u \rangle+ 2\var \langle  \Delta_g u, \eta du \rangle \big) \\
        &\quad -\langle du, \eta  du \rangle - \var  \langle \Delta_g u, \eta \Delta_g u \rangle.
    \end{split}
    \end{equation}
    The last two terms are zero by anti-symmetry of $\eta$ and so the conservation law holds for any fixed $\eta$.\\
    For the only if part of the statement we note that if the conservation law holds for any $\eta$ then, by reversing the steps in \eqref{eq: first_step_solution}, we must have
    \begin{equation*}
        0 = \langle \Delta_g u - \var \Delta_g^2 u , \eta u \rangle
    \end{equation*}
    for any $\eta$. By Lemma \ref{lemma: eta spanning} the possible $\eta u$ span $T_uN$, so we must have \linebreak $(\Delta_g u - \var \Delta^2_gu) \in (T_uN)^\perp$ and so $u$ is $\var$-harmonic.
\end{proof}

Note that if we chose conformal coordinates with factor $e^{2\lambda}$ then the conservation law locally becomes
\begin{align}
         \Div \big( \langle du, \eta u \rangle -  \varepsilon  d (e^{-2\lambda}\langle \Delta u ,  \eta u \rangle)+ 2\var e^{-2\lambda}\langle \Delta u, \eta du \rangle \big) =0   \label{eq: conservation law}.
\end{align}

In \cite{Lamm} Lamm showed that these $\var$-harmonic maps indeed undergo a bubbling procedure, as we will outline.
\begin{lemma} []
    Given a sequence $\varepsilon_k\rightarrow 0$ and $u_{\var_k}$ $\var_k$-harmonic maps with uniformly bounded $\var_k$-energy. Then there exists some subsequence (which we will immediately relabel $\var_k$), $u_0\in C^\infty(M,N)$ harmonic, and some finite set $\Sigma$ such that
    \begin{equation*}
        u_{\var_k}\rightarrow u_0 \quad in \quad C^m_{\text{loc}}(M\backslash\Sigma,N).
    \end{equation*}
\end{lemma}
Using a standard induction argument, as in \cite{lamm+sharp}, we may also assume that there is only one bubble, $\Sigma=\{x_0\}$. We then have the following.
\begin{lemma}[Lemma 3.1 \cite{Lamm}]\label{lemma: epsilon_t_squared}
    There exist
    $t_k\in\R^+$, $t_k\rightarrow0$, $x_k\in M$, $x_k\rightarrow x_0$ and a non-trivial, smooth harmonic map $\omega:S^2\rightarrow N$ such that
    \begin{equation*}
        \frac{\var_k}{t_k^2}\rightarrow0,
    \end{equation*}
    \begin{equation*}
        u_{\var_k}(x_k+t_k\cdot)\rightarrow\omega \;\;  in  \quad C^m_{\text{loc}}(\R^2,N) \quad \forall m \in \N.
    \end{equation*}
\end{lemma}
In this case, the energy identity
\begin{equation*}
    \lim_{k\rightarrow \infty}\tilde{E}_{\var_k}[u_{\var_k}]=E_0[u_0]+E_0[\omega]
\end{equation*}
is equivalent to having
\begin{equation}
    \tilde{E}_{\var_k}[u_{\var_k};B_{R_0}(x_k)\backslash B_{Rt_k}(x_k)]\rightarrow0 \label{eq: energy we want to show}
\end{equation}
as $k\rightarrow \infty,R\rightarrow\infty,R_0\rightarrow0$.\\

We will now recall some more results on $\var$-harmonic maps from \cite{Lamm}. First, we have a regularity result.
\begin{lemma}[Corollary 2.10 \cite{Lamm}]
    There exist $\delta_0>0$ and $c>0$ such that if $u_\var\in C^\infty(M,N)$ is $\var$-harmonic such that for some $x_0\in M, R>0$ we have $\tilde{E}_\var(u_\var,B_{32R}(x_0))<\delta_0$. Then for all $\var>0$ sufficiently small and any $k\in \N$
    \begin{equation*}
        \sum_{i=1}^kR^i||\nabla^iu_\var||_{L^\infty(B_R(x_0))}\leq c\sqrt{\tilde{E}_\var(u_\var,B_{32R}(x_0))}.
    \end{equation*}
\end{lemma}
Now fix $0<\delta<\delta_0$.
Then, as a consequence of there only existing one bubble, there exist $R_0>0$ small enough, $R>0$ large enough, and $k_1>0$ large enough such that for all $k\geq k_1$ we have
\begin{align*}
    \tilde{E}_{\var_k}[u_{\var_k}; B_{2r}\setminus B_r] < \delta 
\end{align*} for every $r\in [Rt_k, \frac{R_0}{2}]$.
Combining these facts gives
\begin{align}
    \sum_{i=1}^4 \vert x\vert^i \vert \nabla^i u_{\var_k} \vert (x) \leq c \sqrt{\delta} \quad \quad \forall  \; 2Rt_k \leq \vert x \vert \leq \frac{R_0}{4}. \label{eq: pointwise_delta_estimate}
\end{align}
This immediately gives us a bound for the $L^{2,\infty}$-norm
\begin{align*}
    \Vert \nabla u_{\var_k} \Vert_{L^{2,\infty}(B_{\frac{R_0}{4}} \setminus B_{2Rt_k})} \leq c\sqrt{\delta} \;  \Big\Vert \frac{1}{\vert x \vert} \Big\Vert_{L^{2,\infty}(\R^2)} \leq c\sqrt{\delta}.
\end{align*}
It will now be sufficient to find a uniform bound on the $\lorentz$-norm of $\nabla u_{\var_k}$. $L^{2,\infty}$ is the dual space of $\lorentz$, so if we have $\Vert \nabla u_{\var_k} \Vert_{\lorentz(B_{\frac{R_0}{8}} \setminus B_{4Rt_k}, N)}\leq c$, we can conclude
\begin{align}
    E_0[u_{\var_k}; B_{\frac{R_0}{8}}\setminus B_{4Rt_k}]
    &\leq c \Vert \nabla u_{\var_k} \Vert_{\lorentz(B_{\frac{R_0}{8}} \setminus B_{4Rt_k}, N)} \cdot
    \Vert \nabla u_{\var_k} \Vert_{L^{2,\infty}(B_{\frac{R_0}{8}} \setminus B_{4Rt_k},N)} \nonumber \\
    &\to 0. \label{eq: dirichlet_vanish}
\end{align} 
The $\var_k$-term of the energy will also vanish on this annulus, as
\begin{align}
    \var_k \int_{B_{\frac{R_0}{4}} \setminus B_{2Rt_k}} \vert \Delta u_{\var_k} \vert^2 \leq c \delta \var_k \int_{B_{\frac{R_0}{4}} \setminus B_{2Rt_k}} \frac{1}{\vert x \vert^4} \leq \frac{c\delta}{R^2}\frac{\var_k}{t_k^2} \to 0  \label{eq: int_vanishes}
\end{align}
using Lemma \ref{lemma: epsilon_t_squared} and \eqref{eq: pointwise_delta_estimate}.
Combining \eqref{eq: dirichlet_vanish} and \eqref{eq: int_vanishes} shows that
\begin{align*}
    \tilde{E}_{\var_k}[u_{\var_k}; B_{\frac{R_0}{8}} \setminus B_{4Rt_k}] \to 0 . \label{eq: vanishing_on_annulus_summing_up}
\end{align*}
So, it remains to show that the $\lorentz$ norm for the gradient of $u_{\var_k}$ is bounded on this annulus.

\bigskip

We first mention two results that will be used to bound this $L^{2,1}$-norm.
\begin{lemma}[Wente inequality, \cite{Coifman1993CompensatedCA, Wente_inequality}]\label{lemma: Wente for L^2,1}
    If $f,g\in W^{1,2}(B)$ and $u$ is the unique solution in $W^{1,2}_0(B)$ to $\Delta u=\nabla f\nabla^\bot g$, then $\Delta u$ is in $\mathcal{H}^1{(\R^2)}$ with
    \begin{equation*}
        ||\nabla u||_{L^{2,1}}\leq c||\nabla f\nabla^\bot g||_{\mathcal{H}^1}\leq c||\nabla f||_{L^2}||\nabla g||_{L^2}.
    \end{equation*}
\end{lemma}

\begin{lemma}[Lemma 2.4, \cite{noneck_li_zhu}]\label{lemma: L^2,1 laplace div bound}
   If $F,u\in C_c^\infty(\R^2)$ satisfy $\Delta u=\Div(F)$, then
   \begin{equation*}
       ||\nabla u||_{L^{2,1}}\leq c||F||_{L^{2,1}} .\label{eq: lemma2.4_Li_zhu}
   \end{equation*}
\end{lemma}

Now we start by constructing an estimate of $u_{\var_k}$ as follows.
\begin{equation}
    \tilde{u}_{\var_k}(x)= \varphi_{\frac{R_0}{16}}(x) \big( (1-\varphi_{4R t_k}(x)) (u_{\var_k}(x)- \Bar{u}_{\var_k}^2) + \Bar{u}_{\var_k}^2- \Bar{u}_{\var_k}^1 \big),
    \label{eq: tilde u definition}
\end{equation}
where we have fixed some $\varphi \in C_c^\infty(B_2)$ such that $\varphi=1$ on $B_1$, taken $\varphi_t(x)= \varphi(\frac{x}{t})$ and have set 
\begin{align*}
    \Bar{u}_{\var_k}^1 &= \frac{1}{|B_{\frac{R_0}{8}} \setminus B_{\frac{R_0}{16}}|} \int\limits_{B_{\frac{R_0}{8}} \setminus B_{\frac{R_0}{16}}} u_{\var_k}(x) \; dx, \\
     \Bar{u}_{\var_k}^2 &= \frac{1}{|B_{8R t_k} \setminus B_{4R t_k }|} \int\limits_{B_{8R t_k} \setminus B_{4R t_k}} u_{\var_k}(x) \; dx .
\end{align*}
Then one can calculate
\begin{align}
    \nabla\tilde{u}_{\var_k} &=\varphi_{\frac{R_0}{16}}  (1-\varphi_{4R t_k})\nabla u_{\var_k} 
  + \nabla (\varphi_{\frac{R_0}{16}}) ( u_{\var_k}  
    - \Bar{u}_{\var_k}^1)  \nonumber \\
    &\quad \quad -\nabla (\varphi_{4R t_k}) ( u_{\var_k} - \Bar{u}_{\var_k}^2). \label{eq: nabla tilde u definition}
\end{align}
We note that the support of $\nabla \varphi_{\frac{R_0}{16}}$ is contained in $B_{\frac{R_0}{8}}\backslash B_{\frac{R_0}{16}}$. Then, using \eqref{eq: pointwise_delta_estimate} to get a bound on $\nabla u_{\var_k}$, we get for $k$ sufficiently large
\begin{align}
    |u_{\var_k}(x) - \Bar{u}_{\var_k}^1| &=\frac{1}{|B_{\frac{R_0}{8}} \setminus B_{\frac{R_0}{16}}|} \int\limits_{B_{\frac{R_0}{8}} \setminus B_{\frac{R_0}{16}}} |u_{\var_k}(x)-u_{\var_k}(y)| \; dy \nonumber \\
    & \leq c\sqrt{\delta} \label{eq: bound_mean_values1} 
\end{align}
for $x\in B_{\frac{R_0}{8}} \setminus B_{\frac{R_0}{16}}$.
Similarly, we have
\begin{align}
    |u_{\var_k}(x) - \Bar{u}_{\var_k}^2|
    &=
   \frac{1}{|B_{8R t_k} \setminus B_{4R t_k }|} \int\limits_{B_{8R t_k} \setminus B_{4R t_k }} |u_{\var_k}(x) -u_{\var_k}(y)| \;dy \nonumber  \\
    &\leq c\sqrt{\delta} \label{eq: bound_mean_values2}
\end{align}
 for $x\in B_{8Rt_k}\backslash B_{4Rt_k}$.
This then gives
\begin{equation}\label{eq: nabla_tilde_u_bound}
    ||\nabla\tilde{u}_{\var_k}||_{L^2(\R^2)}
    \leq||\nabla u_{\var_k}||_{L^2(B_{\frac{R_0}{8}} \setminus B_{4Rt_k })}+c\sqrt{\delta}
\end{equation}
for $k$ sufficiently large.
Note that in two dimensions the $L^{2,1}$-norm of $\nabla \varphi_{\rho}$ is independent of $\rho$ and therefore bounded. This is easily checked by using
\begin{align}
     \big\vert \{x: \vert \nabla \varphi_\rho (x)\vert \geq t  \} \big\vert^{\frac{1}{2}} =
     \big\vert\{x: \big\vert (\nabla \varphi)\Big(\frac{x}{\rho }\Big)\big\vert \geq \rho t  \}\big\vert^{\frac{1}{2}} 
     = \rho   \big\vert\{x: \vert \nabla \varphi(x)\vert \geq \rho t  \}\big\vert^{\frac{1}{2}}
     \label{eq: bounded_gradient_varphi}
\end{align}
and the $L^{2,1}$-norm definition.\\

We now note the fact that our conservation law \eqref{eq: conservation law} holds on $B_{R_0}$ so there exists $L_k\in C^\infty({B_{R_0}})$ with 
\begin{equation}\label{eq: L definition}
    \nabla ^\perp L_k = \langle \nabla u_{\var_k}, \eta u_{\var_k} \rangle -  \varepsilon_k  \nabla \langle e^{-2\lambda}\Delta u_{\var_k} ,  \eta u_{\var_k} \rangle+ 2\var_k \langle e^{-2\lambda} \Delta u_{\var_k}, \eta\nabla u_{\var_k} \rangle.
\end{equation}
Now by the Hodge decomposition there exist $P_k,Q_k\in C^\infty_c({\R^2})$ with
\begin{equation}\label{eq: P,Q definition}
    \langle \nabla \tilde{u}_{\var_k}, \eta u_{\var_k} \rangle = \nabla P_k +\nabla^\perp Q_k.
\end{equation}
Taking the curl of \eqref{eq: P,Q definition} gives the div-curl relationship
\begin{align*}
    \Delta Q_k & = \text{curl}\langle \nabla \tilde{u}_{\var_k}, \eta u_{\var_k} \rangle \nonumber  \\ 
    & = \langle \nabla \tilde{u}_{\var_k}, \nabla^\perp( \eta u_{\var_k}) \rangle.
\end{align*}
Using the same construction as \eqref{eq: tilde u definition} on a larger annulus we can construct $\hat{u}_{\var_k}\in C_c^\infty(\R^2,\R^l)$ such that
\begin{align*}
    \nabla\hat{u}_{\var_k} & =\nabla u_{\var_k} \quad \text{ on }B_{\frac{R_0}{8}}\backslash B_{4Rt_k},\\
    ||\nabla\hat{u}_{\var_k}||_{L^2(\R^2)}
    &\leq c||\nabla u_{\var_k}||_{L^2(B_{\frac{R_0}{4}}\backslash B_{2Rt_k})}+c\sqrt{\delta}\leq c
\end{align*}
using the uniform bound on $\var$-energy for the final inequality.
Then \linebreak 
$\Delta Q=\langle \nabla \tilde{u}_{\var_k},   \nabla^\perp  (\eta\hat{u}_{\var_k} )\rangle$
and so using \eqref{eq: nabla_tilde_u_bound} we get
\begin{align}
    \nonumber||\nabla Q_k||_{L^{2,1}{(\R^2)}}\leq &
    c||\nabla \tilde{u}_{\var_k}||_{L^2(\R^2)}\cdot
    ||\nabla(\eta \hat{u}_{\var_k})||_{L^2(\R^2)}\\
    \leq & c||\nabla u_{\var_k}||_{L^2(B_{\frac{R_0}{8}} \setminus B_{4Rt_k })} + c\sqrt{\delta}. \label{eq: DQ bound}
\end{align}

Now take the divergence of \eqref{eq: P,Q definition} and use \eqref{eq: nabla tilde u definition} to get
\begin{align*}
    \Delta P_k  & = \Div \langle \nabla \tilde{u}_{\var_k}, \eta u_{\var_k} \rangle \nonumber \\
    &= \Div\Big( \varphi_{\frac{R_0}{16}}  (1-\varphi_{4R t_k})\big
    ( \nabla ^\perp L_k
    \nonumber  +\varepsilon_k  \nabla \langle e^{-2\lambda}\Delta u_{\var_k} ,  \eta u_{\var_k} \rangle\\
     &\qquad\qquad\qquad\qquad\qquad\quad- 2\var_ke^{-2\lambda} \langle \Delta u_{\var_k}, \nabla(\eta u_{\var_k}) \rangle \big)  \nonumber\\
    & \quad + \nabla (\varphi_{\frac{R_0}{16}}) \langle u_{\var_k} -\Bar{u}_{\var_k}^1,\eta u_{\var_k}  \rangle  -\nabla (\varphi_{4R t_k}) \langle u_{\var_k} - \Bar{u}_{\var_k}^2,\eta u_{\var_k} \rangle  \Big).
\end{align*}
Now breaking up the terms we define $(\Phi_i)_{i=1,2,3}$ to be the unique solutions in $C^\infty_c({\R^2})$ of
\begin{align}
    \Delta\Phi_1=&
    \Div\big( \varphi_{\frac{R_0}{16}}  (1-\varphi_{4R t_k}) \nabla ^\bot L_k\big),  \\
    \Delta\Phi_2=&
        \varepsilon_k\Div\big( \varphi_{\frac{R_0}{16}}  (1-\varphi_{4R t_k})\big(  \nabla \langle e^{-2\lambda}\Delta u_{\var_k} ,  \eta u_{\var_k} \rangle  \\ 
         &\hspace{12mm}- 2e^{-2\lambda} \langle \Delta u_{\var_k}, \eta \nabla u_{\var_k} \rangle \big) \big),\nonumber\\
    \Delta\Phi_3=&
    \Div\big( \nabla (\varphi_{\frac{R_0}{16}}) \langle u_{\var_k} - \Bar{u}_{\var_k}^1,\eta u_{\var_k} \rangle   \\
    & \hspace{9mm} -\nabla (\varphi_{4R t_k}) \langle u_{\var_k} - \Bar{u}_{\var_k}^2,\eta u_{\var_k} \rangle  \big). \nonumber
\end{align}
We can now bound these all individually, first we have
\begin{align}
    \Delta\Phi_1=
    \nabla(\varphi_{\frac{R_0}{16}}  (1-\varphi_{t_k R})) \nabla ^\perp L_k \label{eq: phi_1_divcurl}
\end{align}
giving us a div-curl relationship.
We define $\tilde{L}_k$ similarly to $\tilde{u}$ in \eqref{eq: tilde u definition} by
\begin{align*}
    \nabla\tilde{L}_{k}(x) &=\varphi_{\frac{R_0}{8}}(x)  (1-\varphi_{2R t_k }(x))\nabla L_k(x) 
  + \nabla (\varphi_{\frac{R_0}{8}}(x)) ( L_{k}(x)  
    - \Bar{L}_{k}^1)  \nonumber \\
    &\quad \quad -\nabla (\varphi_{2R t_k }(x)) ( L_k(x) - \Bar{L}_{k}^2)
\end{align*}
where, similarly to the construction of $\tilde{u}_{\var_k}$, we set $\Bar{L}_{k}^1,\Bar{L}_{k}^2$ to be the mean of $L_k$ on $B_{\frac{R_0}{4}} \setminus B_{\frac{R_0}{8}}$ and $B_{4Rt_k} \setminus B_{2Rt_k}$ respectively. This gives us for $k$ sufficiently large
\begin{align}
    \nabla\tilde{L}_{k} & =\nabla{L}_{k} \quad on \quad  B_{\frac{R_0}{8}}\backslash B_{4Rt_k}, \nonumber\\
    ||\nabla\tilde{L}_{k}||_{L^2(\R^2)}
    &\leq c||\nabla{L}_{k}||_{L^2(B_{\frac{R_0}{4}}\backslash B_{2Rt_k})}+c\sqrt{\delta}.\label{eq: L_tilde_bound}
\end{align}
Where we used \eqref{eq: pointwise_delta_estimate}, \eqref{eq: L definition}, Lemma \ref{lemma: epsilon_t_squared} and the fact that $||e^{-2\lambda(x)}||_\infty$ and $||De^{-2\lambda(x)}||_\infty$ are uniformly bounded to get $|\nabla{L}_{k}(x)|\leq \frac{c\sqrt{\delta}}{|x|}$ on $B_{\frac{R_0}{4}}\backslash B_{2Rt_k}$ for $k$ large enough. The inequality then follows analogously to $\eqref{eq: nabla_tilde_u_bound}$.
We also note, using Lemma \ref{lemma: epsilon_t_squared} and \eqref{eq: L definition}, that
\begin{equation}
    ||\nabla L_k||_{L^2{(B_{\frac{R_0}{4}}\backslash B_{2Rt_k})}}\leq c||\nabla u_{\var_k}||_{L^2{(B_{\frac{R_0}{4}} \setminus B_{2Rt_k })}}+c\sqrt{\delta} \label{eq: L is L^2 bounded}
\end{equation}
for $k$ large enough.
Now using Lemma \ref{lemma: Wente for L^2,1}, \eqref{eq: phi_1_divcurl}, \eqref{eq: L_tilde_bound} and \eqref{eq: L is L^2 bounded} gives
\begin{equation}
\begin{split}
    ||\nabla \Phi_1||_{L^{2,1}{(\R^2)}}
    \leq &||\nabla(\varphi_{\frac{R_0}{16}}(x)  (1-\varphi_{4Rt_k}(x)))||_{L^2{(\R^2)}}\cdot||\nabla \tilde{L}_k||_{L^2{(\R^2)}}\\
    \leq& c ||\nabla L_k||_{L^2{(B_{\frac{R_0}{4}}\backslash B_{2Rt_k})}}+c\sqrt{\delta}\\
    \leq &c||\nabla u_{\var_k}||_{L^2{(B_{\frac{R_0}{4}} \setminus B_{2Rt_k })}}+c\sqrt{\delta}  \label{eq: phi_1}
\end{split}
\end{equation}
for $k$ large enough. 

For the $\Phi_2$ term we note that $\varphi_{\frac{R_0}{16}}(x)  (1-\varphi_{4Rt_k}(x))\leq 1$ and so using Lemma \ref{lemma: L^2,1 laplace div bound}, $\eqref{eq: pointwise_delta_estimate}$ and the bounds on the conformal factor gives
\begin{align*}
    ||\nabla \Phi_2||_{L^{2,1}{(\R^2)}}
    &\leq c\var_k|| \nabla (e^{-2\lambda}\langle\Delta \nonumber u_{\var_k} ,  \eta u_{\var_k} \rangle)- 2e^{-2\lambda} \langle \Delta u_{\var_k}, \eta \nabla u_{\var_k} \rangle||_{L^{2,1}{(B_{\frac{R_0}{8}}\backslash B_{4Rt_k})}} \\
    &\leq
    c\var_k \left\| \frac{1}{|x|^3}\right\|_{L^{2,1}{(B_{\frac{R_0}{8}}\backslash B_{4Rt_k})}}.
\end{align*}
Now we wish to calculate this Lorentz norm.
\begin{lemma}
    $\displaystyle \Big\Vert\frac{1}{|x|^3}\Big\Vert_{L^{2,1}{(B_\alpha\backslash B_\beta)}}=
    c\frac{1}{\beta^2}\sqrt{1-\frac{\beta^2}{\alpha^2}}$
    for some constant $c$.
    \begin{proof}
        For ease of notation set $\displaystyle f(x)=\frac{1}{|x|^3}$.
        First compute the distribution function $\lambda_f$.
        \begin{align*}
        \lambda_f (s)=
            \begin{cases}
                  \pi(\alpha^2-\beta^2) & \quad s < \frac{1}{\alpha^3} \\
                     \pi(s^{-\frac{2}{3}}-\beta^2) & \quad  \frac{1}{\alpha^3} < s <  \frac{1}{\beta^3} \\    
                     0 & \quad\frac{1}{\beta^3} < s \\
            \end{cases}
        \end{align*}
        this then gives
        \begin{align*}
            f^*(t) = 
           \begin{cases}
                     \big(\frac{t}{\pi}+\beta^2\big)^{-\frac{3}{2}} & \quad t < \pi(\alpha^2-\beta^2) \\
                     0 & \quad  \pi(\alpha^2-\beta^2) < t  .\\  
                \end{cases}
        \end{align*}
        Finally
        \begin{align*}
            ||f||_{L^{2,1}(B_\alpha\backslash B_\beta)}&=
            \int\limits_0^{\pi(\alpha^2-\beta^2)} t^{-\frac{1}{2}} \big(\frac{t}{\pi}+\beta^2\big)^{-\frac{3}{2}} \; dt 
            \; = \; \frac{2\pi^{\frac{1}{2}}}{\beta^2}\bigg[\frac{t^{\frac{1}{2}}}{(t+\beta^2)^{\frac{1}{2}}}\bigg]^{\alpha^2-\beta^2}_0 \nonumber \\
            &=\frac{2\pi^{\frac{1}{2}}}{\beta^2}\sqrt{1-\frac{\beta^2}{\alpha^2}}.
        \end{align*}
        
    \end{proof}
\end{lemma}
So using Lemma \ref{lemma: epsilon_t_squared} we get $\lim\limits_{k\rightarrow\infty}||\nabla \Phi_2||_{L^{2,1}{(\R^2)}}=0$. In particular we may assume
\begin{equation}
    ||\nabla \Phi_2||_{L^{2,1}{(\R^2)}} \leq \sqrt\delta
     \label{eq: phi_2}
\end{equation}
 for $k$ large enough.
 
Finally, we bound the $\Phi_3$ term. Note that Lemma \ref{lemma: L^2,1 laplace div bound} along with \eqref{eq: bound_mean_values1}, \eqref{eq: bound_mean_values2} and \eqref{eq: bounded_gradient_varphi} implies
\begin{align}
    ||\nabla \Phi_3||_{L^{2,1}{(\R^2)}} 
    &\leq \Vert \langle(u_{\var_k}- \Bar{u}_{\var_k}^1), \eta u_{\var_k}\rangle \nabla \varphi_{\frac{R_0}{16}} \Vert_{L^{2,1}(\R^2)} \nonumber \\
    & \quad+
    \Vert \langle(u_{\var_k}- \Bar{u}_{\var_k}^2), \eta u_{\var_k}\rangle \nabla \varphi_{4Rt_k } \Vert_{L^{2,1}(\R^2)} \nonumber \\
    &\leq c \sqrt{\delta} \big( \Vert \nabla \varphi_{\frac{R_0}{16}} \Vert_{L^{2,1}(\R^2)}+ \Vert \nabla \varphi_{4R t_k } \Vert_{L^{2,1}(\R^2)} \big)\nonumber\\ 
    &\leq c\sqrt{\delta}. \label{eq: phi_3}
\end{align}

Now as we have $\Delta P_k=\Delta \Phi_1+\Delta \Phi_2+\Delta \Phi_3$, and all functions are compactly supported, we must have $P_k=\Phi_1+\Phi_2+\Phi_3$.
This then gives
\begin{align}
    ||\nabla P_k||_{L^{2,1}(\R^2)}\leq||\nabla \Phi_1||_{L^{2,1}(\R^2)}+||\nabla \Phi_2||_{L^{2,1}(\R^2)}+||\nabla \Phi_3||_{L^{2,1}(\R^2)}. \label{eq: gradient_p}
\end{align}
So overall using the Hodge decomposition \eqref{eq: P,Q definition}, and the equations  \eqref{eq: DQ bound}, \eqref{eq: phi_1}, \eqref{eq: phi_2}, \eqref{eq: phi_3} and \eqref{eq: gradient_p} we achieve, for $k$ large enough
\begin{align}
    ||\langle \nabla \tilde{u}_{\var_k}, \eta u_{\var_k} \rangle||_{L^{2,1}(\R^2)} \leq&
    ||\nabla P_k||_{L^{2,1}(\R^2)} + ||\nabla Q_k||_{L^{2,1}(\R^2)}\nonumber\\
    \leq& c||\nabla u_{\var_k}||_{L^2(B_{\frac{R_0}{4}} \setminus B_{2Rt_k }))}+c\sqrt{\delta}. \label{eq: L21_norm_finally}
\end{align}
Noting here that the constant $c$ potentially depends on $\eta$.

Now using the definition of $\tilde{u}_{\var_k}$, \eqref{eq: tilde u definition}, we note that on $B_{\frac{R_0}{16}}\backslash B_{8Rt_k}$ we have $\nabla\tilde{u}_{\var_k}(x)=\nabla u_{\var_k}(x)$, which by definition always lies in $T_{u_{\var_k}(x)}N$.
Using Lemma \ref{lemma: eta spanning} we may take some finite spanning collection of $\eta_i$, which then gives
\begin{equation}\label{eq: bound by eta collection}
    |\nabla u_{\var_k}|\leq c\sum_i|\langle \nabla\tilde{u}_{\var_k}, \eta_i u_{\var_k} \rangle|\quad \text{on}\;B_{\frac{R_0}{16}}\backslash B_{8Rt_k}.
\end{equation}
So, using the assumption that the $\var$-energy is uniformly bonded, \eqref{eq: L21_norm_finally} and \eqref{eq: bound by eta collection} we get
\begin{align}\label{eq: final L^2,1 bound}
    || \nabla u_{\var_k}||_{L^{2,1}(B_{\frac{R_0}{16}}\backslash B_{8Rt_k})}
    &\leq
    c\sum_i||\langle \nabla\tilde{u}_{\var_k}, \eta_i u_{\var_k} \rangle||_{L^{2,1}(B_{\frac{R_0}{16}}\backslash B_{8Rt_k})}\\\nonumber
    &\leq
    c\sum_i||\langle \nabla\tilde{u}_{\var_k}, \eta_i u_{\var_k} \rangle||_{L^{2,1}(\R^2)}\\\nonumber
    &\leq c||\nabla u_{\var_k}||_{L^2(B_{\frac{R_0}{4}} \setminus B_{2Rt_k }))}+c\sqrt{\delta}\\\nonumber
   &\leq c
\end{align}
which completes the energy bound on the annulus by our prior discussion.\\

\section{No neck property for $\var$-harmonic maps}

The no-neck property will now follow easily from the energy identity.
Because of the assumption of there only being one bubble, as in \cite{noneck_li_zhu}, all we need to prove is that
\begin{align}
    \lim_{\var_k \to 0} \big\Vert u_{\var_k}(\cdot) - u_0(\cdot) - \big( \omega \Big(\frac{\cdot - x_k}{t_k} \Big)  - \omega(\infty) \big)
\big\Vert_{L^\infty} = 0.  \label{eq: noneck-prop}
\end{align}
This then reduces to showing that there is no oscillation on the neck region, specifically we need
\begin{align*}
    \lim_{R_0 \to 0}\lim_{R \to \infty}\lim_{k \to \infty}  \sup_{x,y \in B_{R_0}\setminus B_{R t_k}}\vert u_{\var_k}(x) - u_{\var_k}(y) \vert = 0.
\end{align*}
Then we note that, for $\tilde{u}_{\var_k}$ defined in \eqref{eq: tilde u definition},
\begin{align*}
    \sup_{x,y \in B_{\frac{R_0}{16}}\setminus B_{8R t_k}} \vert u_{\var_k}(x) - u_{\var_k}(y) \vert &\leq \sup_{x,y \in \R^2} \vert \Tilde{u}_{\var_k}(x) - \Tilde{u}_{\var_k}(y) \vert  \\
    &\leq c\Vert \Tilde{u}_{\var_k} \Vert_{C^0(\R^2)} \nonumber \\
    &\leq c \Vert \nabla \Tilde{u}_{\var_k} \Vert_{L^{2,1}(\R^2)}, \nonumber
\end{align*}
where we have used the below result.
\begin{lemma}[Theorem 3.3.4 \cite{Hélein_2002}]
    If $u\in W^{1,2}(\R^2)$ with compact support and $\nabla u\in L^{2,1}(\R^2)$, then $u\in C^0(\R^2)$ with
    \begin{equation*}
        ||u||_{C^0(\R^2)}\leq c||\nabla u||_{L^{2,1}(\R^2)}.
    \end{equation*}
\end{lemma}
So we now only need to show $||\nabla \tilde{u}_{ \var_k}||_{L^{2,1}{(\R^2)}}\rightarrow0$ to get the property \eqref{eq: noneck-prop}.\\

This follows almost immediately from the proof of the energy bound. It remains to bound the normal part of $\nabla \tilde{u}_{ \var_k}$. Define $(\nabla \tilde{u}_{ \var_k})^\top$ and $(\nabla \tilde{u}_{ \var_k})^\bot$ to be the projections of $\nabla \tilde{u}_{ \var_k}(x)$ onto $T_{u(x)}N$ and $\mathcal{N}_{u(x)}N$ respectively.
Then using Lemma \ref{lemma: eta spanning} to take a spanning collection $\eta_i$ and the bound \eqref{eq: L21_norm_finally} we get
\begin{align}\label{eq: tangential bound}
    ||(\nabla\tilde{u}_{\var_k})^\top||_{L^{2,1}(\R^2)}\leq &
    c\sum_i||\langle \nabla\tilde{u}_{\var_k}, \eta_i u_{\var_k} \rangle||_{L^{2,1}(\R^2)}\\
    \leq & c||\nabla u_{\var_k}||_{L^2(B_{\frac{R_0}{4}} \setminus B_{2Rt_k }))}+c\sqrt{\delta} \nonumber
\end{align}
for $k$ sufficiently large. Using \eqref{eq: nabla tilde u definition} we see that
\begin{align*}
    (\nabla\tilde{u}_{\var_k})^\bot=\big(\nabla (\varphi_{\frac{R_0}{16}}(x))\big)^\bot ( u_{\var_k}(x)  
    - \Bar{u}_{\var_k}^1)  -\big(\nabla (\varphi_{4Rt_k}(x))\big)^\bot ( u_{\var_k}(x) - \Bar{u}_{\var_k}^2)
\end{align*}
which we bound, identically to \eqref{eq: phi_3}, using \eqref{eq: bound_mean_values1}, \eqref{eq: bound_mean_values2} and \eqref{eq: bounded_gradient_varphi}
\begin{equation}
    ||(\nabla\tilde{u}_{\var_k})^\bot||_{L^{2,1}(\R^2)}\leq c\sqrt{\delta} \label{eq: normal_part_end}
\end{equation}
for $k$ sufficiently large. So combining the tangential, \eqref{eq: tangential bound}, and normal, \eqref{eq: normal_part_end}, estimates gives
\begin{equation*}
    ||\nabla\tilde{u}_{\var_k}||_{L^{2,1}(\R^2)}\leq c||\nabla u_{\var_k}||_{L^2(B_{\frac{R_0}{4}} \setminus B_{2Rt_k }))}+c\sqrt{\delta}
\end{equation*}
for $k$ sufficiently large.
Then by using the energy identity, \eqref{eq: energy we want to show}, and that $\delta>0$ was arbitrary, the no-neck property follows.

\section{A counterexample to the energy identity for $\var$-harmonic map sequences }

We will now sketch an argument which will extend the counterexample to the energy identity for $\alpha$-harmonic maps, as constructed by Li and Wang in \cite{counterexample}, to $\var$-harmonic maps.
We can use the whole construction in their Section 3 up to their Lemma 3.5. In particular, we use the following. We may construct $N$ as two copies of flat $\mathbb{T}^3$ attached together with a special metric defined by $\psi$ chosen in the neck region. Then a family of maps $u_k\in W^{1,2}(S^2,N)\cap C^0(S^2,N)$ may be constructed, with the properties that they are all pairwise non-homotopic and $\inf\limits_{u\in [u_k]}E_0[u]=8\pi\psi(0)$, where $[u_k]$ is the homotopy class of $u_k$. We also have the result that if $u$ is a non-trivial harmonic map, with $E_0(u)<12\pi\psi(0)$ then $E_0(u)=4\pi\psi(0)$ or $8\pi\psi(0)$.

Now define $u_{\var,k}\in [u_k]\cap W^{2,2}(S^2,N)$, for any $k\in\N,\var>0$, to be a map which minimises the $\var$-energy within the homotopy class. Explicitly it satisfies
\begin{equation*}
    \tilde{E}_\var[u_{\var,k}]=\inf\limits_{u \in [u_k]\cap W^{2,2}(S^2,N)} \tilde{E}_\var[u].
\end{equation*}
Existence follows from $\Tilde{E}_\var$ obeying the Palais-Smale compactness condition.

\begin{lemma}\label{Lemma: lambda equaling sequence}
    For any $\lambda_0> 8\pi\psi(0)$, there exist sequences $\var_k \to 0$ and $i_k\rightarrow \infty$  such that 
    \begin{align*}
        \tilde{E}_{\var_k} [u_{\var_{k,i_k}}] = \lambda_0 \quad \forall \; k \in \N.
    \end{align*}
\end{lemma}
\begin{proof}
First define for $\varepsilon\geq0$
\begin{align*}
    \varphi_k(\var) = \inf\limits_{u \in [u_k]\cap W^{2,2}(S^2,N)} \tilde{E}_\var[u]
\end{align*}
and fix $\var_0>0$. Then we prove that for any fixed $\var \in (0,\var_0)$ the following holds
\begin{align}
    \lim\limits_{k\to \infty } \varphi_k(\var) = + \infty \label{eq: unbounded_infinty}.
\end{align}
Suppose not, then we can find some subsequence $u_{\var,k}$ with bounded $\var$-energy. By using Stokes' theorem we note that that $L^2$-norm of $D^2u_k$ is uniformly bounded, and so indeed $||u_{\var,k}||_{W^{2,2}}$ is uniformly bounded. Then by standard Sobolev embeddings and Morrey's inequality we get $||u_{\var,k}||_{C^{0,\alpha}}$ uniformly bounded for any $0<\alpha<1$ of our choosing. So then by Arzel\`a-Ascoli we may extract a subsequence $u_k\rightarrow u$ which converges in $C^0(S^2,N)$, but then $u_{\var,k}\sim u$ for all $k$ large enough. But by construction all $u_{\var,k}$ are pairwise non-homotopic, giving a contradiction.

Further, we claim that for any fixed $k$ we have $\varphi_k$ continuous on $[0,\var_0)$.
Setting $0<\var_1< \var_2 < \var_0$ we first note that
\begin{align*}
    \varphi_k(\var_1)\leq \tilde{E}_{\var_1}[u_{\var_2,k}]
    \leq  \tilde{E}_{\var_2}[u_{\var_2,k}]
    =  \varphi_k(\var_2)
\end{align*}
so $\varphi_k$ decreases as $\var\rightarrow0$
and so is uniformly bounded for $\var\leq\var_0$. We then have
\begin{align*}
    \varphi_k(\var_1)=&\tilde{E}_{\var_1}[u_{\var_1,k}] \nonumber\\
    =&\tilde{E}_{\var_2}[u_{\var_1,k}]-\frac{\var_2-\var_1}{2}\int\limits_{S^2} \vert \Delta u_{\var_1,k} \vert^2 \nonumber \\
   \geq&\varphi_k(\var_2)-\Big(\frac{\var_2-\var_1}{\var_1}\Big)\Big(\frac{\var_1}{2}\int\limits_{S^2} \vert \Delta u_{\var_1,k} \vert^2\Big)
\end{align*}
giving
\begin{equation*}
    0\leq\varphi_k(\var_2)-\varphi_k(\var_1)\leq c\frac{\var_2-\var_1}{\var_1}
\end{equation*}
which gives continuity on $(0,\var_0)$.
Further, we show that $\varphi_k(\var)$ is right continuous at 0, i.e. $\lim\limits_{\var \searrow 0} \varphi_k(\var) = \varphi_k(0)$. Fix  $u \in W^{2,2}(S^2,N)$ and $\var>0$,
\begin{align*}
    \varphi_k(0)\leq E_0(u)\leq \tilde{E}_\var(u)
\end{align*}
and taking the infimum over $u\in[u_k]\cap W^{2,2}(S^2,N)$ then $\var\rightarrow0$ gives
\begin{align*}
   \varphi_k(0)\leq \lim\limits_{\var \searrow 0} \varphi_k(\var).
\end{align*}
For the other inequality, note $u_k$ is smooth so, for any $\delta>0$, there exists a smooth map $u_k' \in C^\infty(S^2,N)$ in the same homotopy class such that
\begin{align*}
    E_0[u_k'] \leq \varphi_k(0) + \delta.
\end{align*}
Together with $E_0(u_k') = \lim\limits_{\var \searrow 0} \tilde{E}_\var[u_k']$ and the fact that $\varphi_k(\var) \leq \tilde{E}_\var[u_k']$ for each $\var$ we get
\begin{align*}
    \lim\limits_{\var \searrow 0} \varphi_k(\var) \leq  \varphi_k(0) + \delta
\end{align*} so $\varphi$ is right continuous at $0$.
Equation \eqref{eq: unbounded_infinty} implies that given a sequence $\var_k\to 0$, we can take a sequence $i_k\rightarrow\infty$ such that $\varphi_{i_k}(\var_k)> \lambda_0$ for each $k$. Then by continuity of $\varphi$ and the fact that $\varphi_k(0)=8\pi\psi(0)$ for each $k$ we can find some $0<\var_k'\leq\var_k$ such that
\begin{align*}
    \varphi_{i_k}(\var_k') = \tilde{E}_{\var_k'}[u_{\var_k', i_k}] = \lambda_0
\end{align*}
completing the proof.
\end{proof}

Analogously to Li and Wang, we can now construct a counterexample. Take a sequence constructed as in the lemma for some fixed $8\pi\psi(0)<\lambda_0<12\pi\psi(0)$. Then as the maps $u_{\var_k,i_k}$ are pairwise non homotopic, they must blow up as $k\rightarrow\infty$. Set $v_0$ to be the weak limit and $v_1,...,v_l$ the bubbles, these are all smooth homotopic maps with energy less than $12\pi\psi(0)$ so 
\begin{equation*}
    \frac{1}{4\pi\psi(0)}(E_0[v_0]+\sum\limits_i^lE_0[v_i])
\end{equation*}
is an integer. However $\displaystyle \frac{\lambda_0}{4\pi\psi(0)}$ is not, therefore
\begin{equation*}
    \lambda_0=\lim_{k\rightarrow0}\tilde{E}_{\var_k}[u_{\var_k,i_k}]\neq E_0[v_0]+\sum_i^lE_0[v_i]
\end{equation*}
so the energy identity does not hold.

\section{Energy identity and no-neck property for a sequence of Sacks-Uhlenbeck maps to homogeneous spaces}

We will now outline the argument for the $\alpha$-harmonic maps. This is done by generalising Li and Zhu's \cite{noneck_li_zhu} proof of Theorem \ref{theorem: energy+neck Alps} to the case where the target space is a homogeneous space. The only difference is the exact form of the conservation law used.
As discussed earlier, we want to stress that since the $\alpha$-energy is an intrinsic property, the truth of this result will not depend on the existence of an equivariant embedding. However, this also means that the properties of a $\alpha$-harmonic map do not depend on the embedding, so we can freely assume that our target manifold is equivariantly embedded and it will not affect the properties.\\


The idea now is to give a proof of a variant of Noether's Theorem formulated for $\alpha$-harmonic maps using the same $\eta$ as defined in Theorem \ref{theorem: conservation law}. As in their paper, we assume that we are working on the flat unit ball $B$.
\begin{theorem}[Conservation law for $\alpha$-harmonic maps] \label{thm: noether}
    Let $u_\alpha \in C^\infty(B,N)$  be an $\alpha$-harmonic map with $N$ a homogeneous Riemannian manifold equivariantly embedded in $\mathbb{R}^l$ with $\Pi:G\rightarrow O(l)$ the associated Lie group embedding. Let $\eta$ be of the form $\eta = \frac{\partial}{\partial t}\Pi(\gamma(t)) \big|_{t=0}$ for some $\gamma(t)$ a smooth path in $G$ with $\gamma(0)$ the identity. Then the following conservation law holds
    \begin{align}
    \setlength{\abovedisplayskip}{0.5cm}
     \setlength{\belowdisplayskip}{0.5cm}
     \Div\;(F_\alpha  \langle  du_\alpha , \eta u_\alpha \rangle ) &= 0 \label{eq: div_free_field} 
    \end{align}
    where $F_\alpha = (1+\vert \nabla u_\alpha \vert^2)^{\alpha-1}$.
\end{theorem}

\begin{proof}

  Using the facts that $u_\alpha$ is an $\alpha$-harmonic map, so it satisfies \eqref{eq: E-L2}, and that $  \eta u_\alpha \in T_{u_\alpha}N$ we get

    \begin{align*}
       0= &\langle  \Div( F_\alpha d u_\alpha  ) , \eta u_\alpha\rangle   \nonumber\\
       =&   \Div \langle  F_\alpha d u_\alpha, \eta u_\alpha\rangle -  F_\alpha\langle d u_\alpha , \eta du_\alpha \rangle\nonumber\\
        =& \Div(F_\alpha \langle d u_\alpha,  \eta u_\alpha\rangle) 
    \end{align*}
    recalling that $\eta$ is a fixed anti-symmetric matrix.
\end{proof}
Note that in the case of a non equivariant embedding we could have used H\'elein's $\rho$ instead here and this conservation law would still hold, with the proof of energy identity and the no-neck property following identically.

In the case of a round sphere, equipped with the traditional embedding, the isometry group is the whole orthogonal group, so $\eta$ can be any anti-symmetric matrix. In particular, by taking $\eta_{a,b}=\begin{cases}
    1, & (a,b)=(i,j)\\
    -1, & (a,b)=(j,i)\\
    0, & \text{otherwise}\\
  \end{cases}$ we recover the family of conservation laws
  \begin{equation*}
      \Div \big(F_\alpha (u_\alpha^jdu_\alpha^i-u^i_\alpha du^j_\alpha)\big)=0
  \end{equation*}
  obtained by Li and Zhu.

As discussed in Lemma $\ref{lemma: eta spanning}$ we can take a finite collection of $(\eta_{i}q)_{i=1}^I$ spanning $T_qN$ for all $q\in N$. Then by constructing locally and patching over by a partition of unity, or as in Lemma 2 of \cite{Helein_paper}, one can construct $I$ smooth tangent vector fields $(Y_i)_{i=1}^I$ on $N$ such that for any $V\in \Gamma(TN)$ we have
\begin{equation*}
    V(q) = \langle V(q),\eta_{1}q\rangle Y_1(q)+...+ \langle V(q),\eta_{I}q\rangle Y_I(q).
\end{equation*}


So in total we achieve a div-curl relationship,
\begin{align*}
    \Div(F_\alpha \nabla u_\alpha)&=  \Div(F_\alpha \sum\limits_{i=1}^I  \langle \nabla u_\alpha,  \eta_{i}u_\alpha\rangle  Y_i(u_\alpha))  \nonumber \\
    &= \sum\limits_{i=1}^I \Div(F_{\alpha}  \langle \nabla u_\alpha,  \eta_{i}u_\alpha\rangle) Y_i(u_\alpha) + F_{\alpha}  \langle \nabla u_\alpha,  \eta_{i}u_\alpha\rangle \nabla (Y_i(u_\alpha)) \nonumber 
 \\
    &= \sum\limits_{i=1}^I \langle\nabla^\perp G_{\alpha,i} , \nabla (Y_i(u_\alpha))\rangle,
\end{align*}
where $G_{\alpha,i}$ is defined to be a solution in $W^{1,2}(B)$ to $\nabla^\perp G_{\alpha,i} = F_{\alpha}  \langle \nabla u_\alpha,  \eta_{i}u_\alpha\rangle $.
This further allows us to write, for $\tilde{u}_\alpha$ constructed as in step 2 of section 3 in \cite{noneck_li_zhu},
\begin{align*}
     \Div(F_\alpha \nabla \tilde{u}_\alpha)=&\sum\limits_{i=1}^I \langle\nabla^\perp G_{\alpha,i} , \nabla (Y_i(u_\alpha)\varphi_\delta(1-\varphi_{r_\alpha R})\rangle\\
     &+
     \Div\Big(F_\alpha\big( (u_\alpha-\bar{u}_\alpha^1)\nabla\varphi_\delta-( u_\alpha-\bar{u}_\alpha^2)\nabla\varphi_{r_\alpha R}\big)\Big)
\end{align*}
which is analogous to their equation (3.13).

Using these results it is easy to obtain the result of no energy loss and no neck property by following the rest of the proof of Li and Zhu.
The only terms that differ are the exact form of $G_{\alpha,i}$ and in some formula we will have $Y_i(u_\alpha)$ instead of $u_\alpha^i$, however it is easy to see that the required bounds will still follow. Indeed, all the required bounds on $G_{\alpha,i}$ follow from the fact that
$|\nabla G_{\alpha,i}|\leq c|F_\alpha|\cdot|\nabla u_\alpha|$ pointwise, which still holds in our case.


\section{Declarations}
\subsection{Funding} The first author was funded by Deutsche Forschungsgemeinschaft (DFG) - RTG 2229, Project number 281869850. The first author is now funded by the Swiss National Fund, Project SNF $200020\_219429$.\\
The second author is funded by Engineering and Physical Sciences Research Council (EPSRC) - EP/W524372/1, Studentship 2927009.\\
The second author was given funding by Karlsruhe Institute of Technology to visit there for two weeks during this project.\\

\subsection{Data availability} No datasets were generated or analysed while completing this project.\\

\subsection{Conflict of interest} The authors have no conflicts of interest to disclose.\\

\bibliographystyle{abbrv}
\bibliography{Bibliography}

\end{document}